\documentclass[12pt]{article}
\usepackage{amsfonts,amsthm, amsmath}
\usepackage{epsfig}
\usepackage{makeidx}
\usepackage{graphicx,epstopdf}
\DeclareGraphicsExtensions{.ps,.png,.jpg}
\usepackage{enumerate}

\makeindex
\usepackage{subfig}

\newenvironment{namelist}[1]{%
\begin{list}{}
{

\settowidth{\labelwidth}{#1}
\setlength{\leftmargin}{1.1\labelwidth}
}
}{%
\end{list}}
\newcommand{\ncom}{\newcommand}
\ncom{\ul}{\underline}
\ncom{\beq}{\begin{equation}}
\ncom{\eeq}{\end{equation}}
\ncom{\bea}{\begin{eqnarray*}}
\ncom{\eea}{\end{eqnarray*}}
\ncom{\beqa}{\begin{eqnarray}}
\ncom{\eeqa}{\end{eqnarray}}
\ncom{\nno}{\nonumber}
\ncom{\non}{\nonumber}
\ncom{\ds}{\displaystyle}
\ncom{\half}{\frac{1}{2}}
\ncom{\mbx}{\makebox{.25cm}}
\ncom{\hs}{\mbox{\hspace{.25cm}}}
\ncom{\rar}{\rightarrow}
\ncom{\Rar}{\Rightarrow}
\ncom{\noin}{\noindent}
\ncom{\bc}{\begin{center}}
\ncom{\ec}{\end{center}}
\ncom{\sz}{\scriptsize}
\ncom{\rf}{\ref}
\ncom{\s}{\sqrt{2}}
\ncom{\sgm}{\sigma}
\ncom{\Sgm}{\Sigma}
\ncom{\psgm}{\sigma^{\prime}}
\ncom{\dt}{\delta}
\ncom{\Dt}{\Delta}
\ncom{\lmd}{\lambda}
\ncom{\Lmd}{\Lambda}
\ncom{\Th}{\Theta}
\ncom{\e}{\eta}
\ncom{\eps}{\epsilon}
\ncom{\pcc}{\stackrel{P}{>}}
\ncom{\lp}{\stackrel{L_{p}}{>}}
\ncom{\dist}{{\rm\,dist}}
\ncom{\sspan}{{\rm\,span}}
\ncom{\re}{{\rm Re\,}}
\ncom{\im}{{\rm Im\,}}
\ncom{\sgn}{{\rm sgn\,}}
\ncom{\ba}{\begin{array}}
\ncom{\ea}{\end{array}}
\ncom{\hone}{\mbox{\hspace{1em}}}
\ncom{\htwo}{\mbox{\hspace{2em}}}
\ncom{\hthree}{\mbox{\hspace{3em}}}
\ncom{\hfour}{\mbox{\hspace{4em}}}
\ncom{\vone}{\vskip 2ex}
\ncom{\vtwo}{\vskip 4ex}
\ncom{\vonee}{\vskip 1.5ex}
\ncom{\vthree}{\vskip 6ex}
\ncom{\vfour}{\vspace*{8ex}}
\ncom{\norm}{\|\;\;\|}
\ncom{\integ}[4]{\int_{#1}^{#2}\,{#3}\,d{#4}}
\ncom{\vspan}[1]{{{\rm\,span}\{ #1 \}}}
\ncom{\dm}[1]{ {\displaystyle{#1} } }
\ncom{\ri}[1]{{#1} \index{#1}}

\newtheorem{theorem}{\bf Theorem}[section]
\newtheorem{remark}{\bf Remark}[section]
\newtheorem{proposition}{Proposition}[section]

\newtheorem{corollary}{Corollary}[section]

\newtheorem{definition}{Definition}[section]
\newtheoremstyle
    {remarkstyle}
    {}
    {11pt}
    {}
    {}
    {\bfseries}
    {:}
    {     }
    {\thmname{#1} \thmnumber{#2} }

\theoremstyle{remarkstyle}

\begin{document}

\newpage

\begin{center}
{\Large \bf Mittag-Leffler L\'evy Processes}
\end{center}
\vone

\begin{center}
{\bf  Arun Kumar$^*$ and N. S. Upadhye$^{**}$}\\
{\it *Indian Statistical Institute, Chennai Center, Taramani,  Chennai-600036, India\\
and\\
**Department of Mathematics, Indian Institute of Technology Madras, Chennai 600036, India}
\end{center}

\vtwo
\begin{center}
\noindent{\bf Abstract}
\end{center}
In this article, we introduce Mittag-Leffler L\'evy process and provide two alternative representations of this process. First, in terms of Laplace transform of the marginal densities and next as a subordinated stochastic process. Both these representations are useful in analyzing the properties of the process. Since integer order moments for this process are not finite, we obtain fractional order moments. Its density function and corresponding L\'evy measure density is also obtained.  Further, heavy tailed behavior of densities and stochastic self-similarity of the process is demonstrated. Our results generalize and complement the results available on Mittag-Leffler distribution in several directions. 

\vone \noindent{\it Key words:} Mittag-Leffler distribution, subordinated stochastic processes, L\'evy densities.
\vtwo
\setcounter{equation}{0}

\section{Introduction}
In recent years there is an increased attention on Mittag-Leffler (ML) function as well as on ML probability distribution.  ML distribution is a natural generalization of exponential distribution. ML waiting times are used in defining a fractional version of standard Poisson process that is also called fractional Poisson process (see e.g. Meerschaert et al. 2010; Repin and Saichev, 2000; Laskin, 2003; Beghin and Orsingher, 2009).
Let $X$ be a ML distributed random variable (rv) with parameters $0<\alpha\leq 1$ and $\lambda>0.$ Then the Laplace transform (LT) of $X$ is given by (see e.g. Cahoy et al. 2010; Pillai, 1990)
\begin{equation}\label{LT-ML}
\mathbb{E} e^{-uX} = \frac{\lambda}{\lambda + u ^{\alpha}}.
\end{equation}

Pillai (1990) established the infinite divisibility and geometric infinite divisibility for $X$, along with other interesting properties. The ML density function can be written in terms of ML function as
$f_{X}(x;\lambda, \alpha) = 1- E_{\alpha, \lambda}(-x^{\alpha})$, $x >0$, where 
\begin{equation}
E_{\alpha, \lambda}(z) = \frac{1}{2\pi i} \int_{C} \frac{t^{\alpha - \lambda}}{t^{\alpha}-z}dt, 
\end{equation}
with contour $C$ starts and ends at $-\infty$ and circles around the singularities and branch points of the integrand (see e.g. Gorenflo et al. 2014). The infinite series representation of $E_{\alpha, \lambda}(z)$ is given by
\begin{equation}\label{ML-Series}
E_{\alpha, \lambda}(z) = \sum_{k=0}^{\infty} \frac{z^k}{\Gamma(\alpha k+\lambda)}.
\end{equation}

Note that ML distributions are heavy tailed and in the domain of attraction of stable laws (see e.g. Feller, 1971), which can be seen as follows. Let $X_1, X_2,\ldots, X_n$ be iid ML distributed random variables with LT as in \eqref{LT-ML}, then LT of the rescaled rv  $T_n = \lambda^{1/\alpha}n^{-1/\alpha}\sum_{i=1}^{n} X_i$ is given by
\begin{equation}\label{stable-limit}
\mathbb{E}(e^{-uT_n}) = (1+u^{\alpha}/n)^{-n}\rightarrow e^{-u^{\alpha}},\;{as}\;n\rightarrow\infty
\end{equation}
 which is the LT of a stable distribution.

\setcounter{equation}{0}
\section{Mittag-Leffler L\'evy Process (MLLP)}
As mentioned earlier, ML distribution is infinitely divisible and hence we can define a L\'evy process corresponding to this distribution. We characterize the MLLP by defining the LT of its marginal density as follows.
\begin{definition}[MLLP]
A stochastic process $M_{\alpha, \lambda}(t)$ is a MLLP if it is a L\'evy process and its marginal density has the LT given by
\begin{equation}\label{LT-MLLP}
\mathbb{E} e^{-u M_{\alpha, \lambda}(t)} = \left(\frac{\lambda}{\lambda + u ^{\alpha}}\right)^t.
\end{equation}
\end{definition}

Recently, subordinated stochastic processes have attracted much attention due to their applications in fractional partial differential equations, stochastic volatility modeling in finance, and for interesting probabilistic properties.  Next, we derive subordinated stochastic representation of MLLP. 

\noindent
Let $G_{\lambda, \beta}(t)$ be a gamma L\'evy process such that $G_{\lambda, \beta}(t) \sim$ gamma$(\lambda, \beta t)$ with density
\begin{align*}
f_{G_{\lambda, \beta}(t)}(y) = \frac{\lambda^{\beta t}}{\Gamma{(\beta t)}} y^{\beta t-1} e^{-\lambda y}, \;y>0, \;t>0.
\end{align*}

\noindent Further, the LT of $G_{\lambda, \beta}(t)$ is given by
\begin{align} \label{LT-Gamma}
\mathbb{E}e^{-u G_{\lambda, \beta}(t)} = \left(\frac{\lambda}{\lambda + u}\right)^{\beta t}.
\end{align} 
Let $S_{\alpha}(t)$ be a stable L\'evy process with stability parameter $\alpha \in(0,1).$ Note that stable process has strictly increasing sample paths and its tail probability decays polynomially such that $\mathbb{P} (S_{\alpha}(t)>x) \sim d x^{-\alpha},$ as $x\rightarrow\infty$ for some constant $d>0$. Stable processes are self-similar such that $S_{\alpha}(ct) \stackrel{d} = c^{1/\alpha}S_{\alpha}(t),\; c>0$, and the LT of $S_{\alpha}(t)$ is given by

\begin{align} \label{LT-Stable}
\mathbb{E}e^{-u S_{\alpha}(t)} = e^{-tu^{\alpha}}.
\end{align}

\begin{definition}[Alternative representation of MLLP]
Let $S_{\alpha}(t)$ be a stable L\'evy process and $G_{\lambda, \beta}(t)$ be gamma L\'evy process independent of $S_{\alpha}(t)$. A three parameter MLLP can be defined as a subordinated stochastic process by using gamma L\'evy process as a subordinator such that
\begin{equation}\label{MLLP-Sub}
M_{\alpha, \lambda, \beta}(t) :=  S_{\alpha}\left(G_{\lambda, \beta}(t)\right).
\end{equation}
\end{definition}

Note that 
\begin{align*}
\mathbb{E} e^{-u M_{\alpha, \lambda, \beta}(t)} &= \mathbb{E} e^{-u   S_{\alpha}\left(G_{\lambda, \beta}(t)\right)}\\
& = \mathbb{E}\left(\mathbb{E} e^{-u S_{\alpha}\left(G_{\lambda, \beta}(t)\right)}|G_{\lambda, \beta}(t)\right)\\
& = \mathbb{E} \left(e^{-u^{\alpha} G_{\lambda, \beta}(t)}\right)\;\; (\mbox{using}\; \eqref{LT-Stable})\\
& = \left(\frac{\lambda}{\lambda + u^{\alpha}}\right)^{\beta t} \;\; (\mbox{using} \;\eqref{LT-Gamma}),
\end{align*}
which is same as in \eqref{LT-MLLP} for $\beta=1.$
For simplicity, throughout the article, we assume $\beta =1$ and denote $M_{\alpha, \lambda}(t) := M_{\alpha, \lambda, 1}(t)$. 
We have, 
\begin{align}\label{ALR}
M_{\alpha, \lambda}(t): = S_{\alpha}(G_{\lambda,1}(t)) \stackrel{d} = G_{\lambda,1}(t)^{1/\alpha} S_{\alpha}(1), 
\end{align}
using self-similar property of stable processes. Hence ML distribution is a mixture of gamma and stable distributions. It is well known that given a random variable $U$, uniformly distributed on $[0,1]$, the transformed random variable $|\ln U|/\lambda$ is exponentially distributed with mean $\lambda$. Thus, \eqref{ALR} can be used in generating ML random numbers. Let $U_1, U_2$ and $U_3$ be independent and uniformly distributed over [0,1]. Note that for $t=1$, $G_{\lambda, 1}(1) \sim$ exp($\lambda$). Thus by using Kanter (1975) result on stable random number generation and \eqref{ALR}, we have
\begin{equation}
M_{\alpha, \lambda}(1)\stackrel{d} = \frac{\sin(\alpha\pi U_1)[\sin((1-\pi)U_2)]^{1/\alpha-1}}{[\sin(\pi U_1)]^{1/\alpha}|\ln U_2|^{1/\alpha-1}} \frac{|\ln U_3|^{1/\alpha}}{\lambda^{1/\alpha}}.
\end{equation}

\noindent Similar to \eqref{stable-limit}, we have convergence in distribution for MLLP process.
\begin{proposition}
$$
\lim_{t\rightarrow \infty} \frac{\lambda^{1/\alpha}M_{\alpha,\lambda}(t)}{t^{1/\alpha}} \stackrel{d}\rightarrow S_{\alpha}(1).
$$
\end{proposition}
\begin{proof}
We have
\begin{align*}
\frac{\lambda^{1/\alpha}M_{\alpha,\lambda}(t)}{t^{1/\alpha}} &\stackrel{d}= \lambda^{1/\alpha}\left(\frac{G_{\lambda,1}(t)}{t}\right)^{1/\alpha} S_{\alpha}(1)\;\;(\mbox{using }\; \eqref{ALR})\\
& \stackrel{d} \rightarrow \lambda^{1/\alpha}(\mathbb{E}G_{\lambda,1}(1))^{1/\alpha}S_{\alpha}(1) = S_{\alpha}(1)\;\;\mbox{as}\;t\rightarrow\infty, 
\end{align*}
where the last convergence follows from the fact that for a subordinator $T_t$, $\lim_{t\rightarrow\infty}T_t/t\stackrel{a.s.}\rightarrow \mathbb{E}T_1$ (see e.g. Bertion, 1996, p. 92). Further, $\mathbb{E}G_{\lambda,1}(1) = 1/\lambda$ and the result that if $X_n\stackrel{d}\rightarrow X$ and $Y_n\stackrel{d}\rightarrow c$ for some constant $c$, then $X_nY_n\stackrel{d}\rightarrow cX$ as $n\rightarrow\infty$. 
\end{proof}

\setcounter{equation}{0}
\section{MLLP density}
We next obtain the infinite series representation of density function of MLLP. The main idea is to invert the LT given in \eqref{LT-MLLP}.
\begin{proposition}\label{marginal-density}
MLLP has marginal densities in following infinite series form
\begin{align}\label{MLLP-density}
f_{M_{\alpha,\lambda}(t)}(x) = \sum_{k=0}^{\infty} (-1)^k \frac{\lambda^{t+k} \Gamma(t+k)}{\Gamma(t)\Gamma(k+1)}\frac{x^{\alpha(t+k)-1}}{\Gamma(\alpha(t+k))},\; x>0.
\end{align}
\end{proposition}
\begin{proof}
We have
\begin{align*}
\tilde{f}_{M_{\alpha,\lambda}(t)}(u) &= \mathbb{E}e^{-uM_{\alpha,\lambda}(t)} = \left(\frac{\lambda}{\lambda + u^{\alpha}}\right)^t\\
& = \lambda^tu^{-\alpha t} \sum_{k=0}^{\infty}\left(1+\frac{\lambda}{u^{\alpha}}\right)^{-t} = \sum_{k=0}^{\infty} \lambda^{t+k}{-t\choose k} u^{-\alpha(t+k)}\\
& = \sum_{k=0}^{\infty} (-1)^k \lambda^{t+k} \frac{\Gamma(t+k)}{\Gamma(t)\Gamma(k+1)} u^{-\alpha(t+k)}.
\end{align*}
Now, by inverting the LT in both sides and using $\mathcal{L}^{-1}(1/u^a) = x^{a-1}/\Gamma(a)$, $a>0$, where $\mathcal{L}^{-1}$ denotes the inverse LT,  we have the desired result.
\end{proof}

Using the marginal density function of MLLP, we obtain the L\'evy density corresponding to L\'evy-Khintchine representation of this process.

\begin{proposition}\label{levy-density}
The L\'evy density in L\'evy-Khintchine representation of the process is given by
\begin{align} \label{MLLP-LevyMeasure}
\nu(dx) = \frac{\alpha}{x}E_{\alpha,1}(-\lambda x^{\alpha}), x>0.
\end{align} 
\end{proposition}
\begin{proof}
For positive L\'evy process (i.e. subordinator) with probability density $f(x, t)$, L\'evy measure density $\nu(dx)$ is given by (see e.g. Barndorff-Nielsen, 2010)
\begin{equation*}
\nu(dx) = \lim_{t\downarrow 0} \frac{1}{t}f(x,t).
\end{equation*}
Thus, using \eqref{MLLP-density}, we obtain
\begin{align*}
\nu(dx) &= \sum_{k=0}^{\infty}(-1)^k\lambda^k \frac{x^{\alpha k-1}}{\Gamma(\alpha k)} \lim_{t\downarrow 0} \frac{\Gamma(t+k)}{t\Gamma(t)\Gamma(k+1)}\\
&= \sum_{k=0}^{\infty}(-1)^k\lambda^k \frac{x^{\alpha k-1}}{k\Gamma(\alpha k)}\\
& = \frac{\alpha}{x} \sum_{k=0}^{\infty}(-1)^k \frac{(\lambda x^{\alpha})^{k}}{\Gamma(\alpha k+1)} = \frac{\alpha}{x}E_{\alpha,1}(-\lambda x^{\alpha})\; (\mbox{using}\;\eqref{ML-Series}).  
\end{align*}
\end{proof}
\begin{remark}
For $\alpha = 1,$ \eqref{MLLP-LevyMeasure} reduces to $\nu(dx) = x^{-1}e^{-\lambda x}$ and which is the L\'evy density corresponding to the gamma L\'evy procees $G_{\lambda,1}(t)$ (see e.g. Applebaum, 2009, p. 55).
\end{remark}

\begin{remark}
Since, $E_{\alpha,1}(-\lambda x^{\alpha})\rightarrow 1$ as $x\rightarrow 0$, we have $\nu(dx)\sim 1/x$ as $x\rightarrow 0$, where $f(x)\sim g(x)$ as $x\rightarrow x_0$ means that $\lim_{x\rightarrow x_0}f(x)/g(x)=1.$ Hence $\int_{0}^{\infty} \nu(x) dx = \infty.$ Thus sample paths of MLLP  are strictly increasing by an application of Theorem 21.3 of Sato (1999). Similar to gamma L\'evy process the L\'evy measure for MLLP is concentrated at origin and hence this process has an infinite arrival rate of jumps, most of which are small.
\end{remark}

For asymptotic behavior of density function, we use Tauberian theorem. First, we recall that a function $L(t)$ is {\it slowly varying} at some $t_0$, if for all fixed $c>0$, $\lim_{t\rightarrow t_0}\frac{L(ct)}{L(t)}=1$. For the sake of convenience, we state Tauberian theorem here (see e.g. Bertoin, 1996, p. 10).
\begin{theorem}[Tauberian Theorem]\label{Taubarian} Let $L:(0,\infty)\rightarrow(0,\infty)$ be a slowly varying function at $0$ (respectively $\infty$) and let $q\geq 0.$ Then, for a function $U: (0,\infty)\rightarrow(0,\infty)$, with corresponding LT $\tilde{U}$ the following are equivalent: 
$$(i)~~~~~ U(t)\sim t^{q}L(t)/\Gamma(1+ q),~~  t\rightarrow 0 ~(\mbox{respectively}~ t\rightarrow\infty).$$
$$\hspace{-.2cm}(ii)~~~~~~ \tilde{U}(u)\sim u^{-q-1}L(1/q), ~~ q\rightarrow \infty~ (\mbox{respectively}~ q\rightarrow 0).$$
\end{theorem}

\begin{proposition}
The density function of MLLP has following asymptotic behaviors
\begin{align*}
 f_{M_{\alpha, \lambda}(t)}(x) \sim \left\{
\begin{array}{ll}
\displaystyle  \frac{\lambda^t}{\Gamma(\alpha t)} x^{\alpha t-1},~~~~~~~~~~~~~~~\mbox{as}~ x\rightarrow 0,\\
\frac{1}{\Gamma(-\alpha t)} (\lambda + x^{\alpha})^t x^{-\alpha t-1}, ~~~~~~~~~\mbox{as}~x\rightarrow\infty.
\end{array}\right.
\end{align*}
\end{proposition}
\begin{proof}
We have
\begin{align*}
\tilde{f}_{M_{\alpha, \lambda}(t)}(u) &= \mathbb{E} e^{u M_{ \alpha, \lambda}(t)} = \left(\frac{\lambda}{\lambda + u ^{\alpha}}\right)^t\\
&= \lambda^t u^{-\alpha t} \left(1+ \frac{\lambda}{u^{\alpha}}\right)^{-t} \sim \lambda^t u^{-\alpha t}\;\;\mbox{as}\;u\rightarrow\infty.
\end{align*}
Now, by the application of Tauberian theorem, we have $f_{M_{\alpha, \lambda}(t)}(x) \sim \frac{\lambda^t}{\Gamma(\alpha t)} x^{\alpha t-1}$, as $x\rightarrow 0$. Further, we have 
\begin{align*}
\tilde{f}_{M_{\alpha, \lambda}(t)}(u) &= \left(1+ \frac{u^{\alpha}}{\lambda}\right)^{-t}
= \left(1+ \frac{u^{-\alpha}}{\lambda}\right)^{t}\left[ \left(1+ \frac{u^{-\alpha}}{\lambda}\right)\left(1+ \frac{u^{\alpha}}{\lambda}\right)\right]^{-t}\\
& = \left(1+ \frac{u^{-\alpha}}{\lambda}\right)^{t} \left[1+\frac{1}{\lambda^2} + \frac{u^{\alpha}}{\lambda} + \frac{u^{-\alpha}}{\lambda}\right]^{-t}\\
& = \left(1+ \frac{u^{-\alpha}}{\lambda}\right)^{t} \lambda^t u^{\alpha t}\left[ 1+ \left(\lambda+\frac{1}{\lambda}\right)u^{\alpha} + u^{2\alpha}\right]^{-t}\\
&\sim  \left(1+ \frac{u^{-\alpha}}{\lambda}\right)^{t} \lambda^t u^{\alpha t} \;\; \mbox{as} \; u\rightarrow 0\\
&\equiv L(1/u)u^{-(-\alpha t-1)-1}
\end{align*}
Note that the function $L(u) = \left(\lambda+ u^{\alpha}\right)^{t}$ is slowly varying at $0$. Hence again by the application of Tauberian Theorem, we have $f_{M_{\alpha, \lambda}(t)}(x) \sim \frac{1}{\Gamma(-\alpha t)} (\lambda + x^{\alpha})^t x^{-\alpha t-1}$, as $x\rightarrow \infty.$   
\end{proof}

\setcounter{equation}{0}
\section{MLLP fractional order moments}

Using \eqref{MLLP-Sub}, we have $\mathbb{E} (M_{\alpha, \lambda}(t)) =\mathbb{E}\left(\mathbb{E}(S_{\alpha}(G_{\lambda, \beta}(t))|G_{\lambda, \beta}(t))\right) =\infty$, since mean is infinite for $\alpha$-stable distribution.  Thus, we need to find out the fractional moments of the MLLP. Here we use LT definition of MLLP to obtain the fractional order moments.
For a positive random variable $X$, the fractional order moments from LT $\phi(u)$ can be obtained as follows.
\begin{align*}
\int_{0}^{\infty} \frac{d}{du}[\phi(u)] u^{p-1} du &= \int_{0}^{\infty} \frac{d}{du}[\mathbb{E}e^{-uX}] u^{p-1} ds \\
& = \mathbb{E}\left[\int_{0}^{\infty} \frac{d}{du}[e^{-uX}] u^{p-1} du\right]\;\; (\mbox{By Fubini Theorem)} \\
& = (-1) \mathbb{E}\left[X\int_{0}^{\infty} e^{-uX} u^{p-1} du\right] = (-1) \Gamma(p)\mathbb{E}X^{1-p}.
\end{align*}
Thus $(1-p)$-th order moment, where $p$ is a real number such that $p\in(0,1)$ is given by
\begin{equation} \label{final-moment}
\mathbb{E}(X^{1-p}) = \frac{(-1)}{ \Gamma(p)} \int_{0}^{\infty} \frac{d}{du}[\phi(u)] u^{p-1} du.
\end{equation}

\begin{corollary}
The fractional order moments of MLLP are given by
\begin{equation}\label{moments-mllp}
\mathbb{E}(M_{\alpha, \lambda}(t)^q) = \frac{t} {\lambda^q \Gamma(1-q)} B\left(1-\frac{q}{\alpha}, t+\frac{q}{\alpha}\right),\;\; 0<q< \alpha.
\end{equation}
\end{corollary}
\begin{proof}
By taking $p=1-q$ in \eqref{final-moment} and using \eqref{LT-MLLP}, we have
\begin{align*}
\mathbb{E}(M_{\alpha, \lambda}(t)^q) &= \frac{-1}{\Gamma(1-q)} \int_{0}^{\infty} \frac{d}{du}\left[\frac{\lambda}{\lambda + u^{\alpha}}\right]^{t} u^{-q} du\\
& = \frac{\alpha t \lambda^t}{\Gamma(1-q)} \int_{0}^{\infty} \frac{u^{\alpha-q-1}}{(\lambda+u^{\alpha})^{t+1}}du\\
& = \frac{ t \lambda^{-\frac{q}{\alpha}}}{\Gamma(1-q)} \int_{0}^{\infty} \frac{y^{-\frac{q}{\alpha}}}{(1+y)^{t+1}}dy\;\;(\mbox{by substituting}\;u^{\alpha} = \lambda y)\\
& = \frac{t} {\lambda^\frac{q}{\alpha} \Gamma(1-q)} \int_{0}^{\infty} \frac{y^{(1-\frac{q}{\alpha})-1}}{(1+y)^{(1-\frac{q}{\alpha})+(t+\frac{q}{\alpha})}}dy\\
& = \frac{t} {\lambda^\frac{q}{\alpha} \Gamma(1-q)} B\left(1-\frac{q}{\alpha}, t+\frac{q}{\alpha}\right).
\end{align*}
\end{proof}
\begin{remark} Alternatively representation \eqref{ALR} of MLLP can also be used in finding the fractional order moments.
Note that
\begin{align*}
\mathbb{E}\left(M_{\alpha,\lambda}(t)\right)^q & = \mathbb{E}[{G_{\lambda,1}(t)}]^{q/\alpha}\mathbb{E}[S_{\alpha}(1)]^q\\
& = \frac{\Gamma{(\frac{q}{\alpha}+t})}{\lambda^\frac{q}{\alpha} \Gamma{(t)}} \frac{\Gamma{(1-\frac{q}{\alpha})}}{\Gamma{(1-q)}}\\
& =  \frac{t} {\lambda^\frac{q}{\alpha} \Gamma(1-q)} B\left(1-\frac{q}{\alpha}, t+\frac{q}{\alpha}\right) 
\end{align*}
\end{remark}
\begin{remark}
for $\alpha=1$ and $t=1$ in \eqref{moments-mllp}, we have $\mathbb{E}M_{1, \lambda}(1)^q = \frac{1} {\lambda^q \Gamma(1-q)}B(1-q,1+q) = \frac{\Gamma(1+q)}{\lambda^q}$, which is the $q$-th order moment of exponential distribution with parameter $\lambda.$
\end{remark}

\setcounter{equation}{0}
\section{Stochastic self-similarity}
We use the definition of stochastic self-similarity introduced by Kozubowski et al. (2006). For completeness purpose we recall their definition.
\begin{definition}[Stochastic self-similarity]
Let X(t) be a stochastic process. Let $T_{c}(t),c \in(1,\infty)$ be a family of stochastic processes independent of $X(t)$ such that $T_c(0) = 0$ a.s. with non-decreasing sample paths and $\mathbb{E}T_c(t) = ct$. The process $X(t)$ is stochastically self-similar with index $\eta > 0$  with respect to the family $T_c(t)$ if
\begin{equation}
\{X(T_c(t)), t\geq 0\} \stackrel{d}= \{c^{\eta}X(t), t\geq 0\},\;c\in(1,\infty),
\end{equation}
where $\stackrel{d}=$ denotes the equivalence of finite dimensional distributions.
\end{definition}

\noindent Let $N_p(t)$ be a negative binomial L\'evy process with drift defined as $N_{p}(t):= t+ NB_p(t)$, where $NB_p(t)$ is a L\'evy process such that 
\begin{equation}\label{NB-marginal}
\mathbb{P}(NB_p(t) = j) = {t+j-1 \choose j}p^t (1-p)^j,\; j=0,1\cdots,\; t>0.
\end{equation}
We have the following result for the stochastic self-similarity of the MLLP.
\begin{proposition}
MLLP is stochastic self-similar with index $1/\alpha$ with respect to the process $N_{1/c}(t)$. That is 
\begin{equation}\label{stochastic-selfsimilarity}
\{M_{\alpha,\lambda}(N_{1/c}(t)), t\geq 0\} \stackrel{d}= \{c^{1/\alpha} M_{\alpha, \lambda}(t), t\geq 0\}.
\end{equation}
\end{proposition}
\begin{proof}
Since $M_{\alpha,\lambda}(N_{1/c}(t))$ and $c^{1/\alpha} M_{\alpha, \lambda}(t)$ are both L\'evy processes to prove \eqref{stochastic-selfsimilarity}, it is sufficient to show the equivalence of marginal distributions at $t=1.$ Thus, we have
\begin{align*}
\mathbb{E} \left(e^{-u M_{\alpha,\lambda}(N_{1/c}(1))}\right) &= \mathbb{E} \left(e^{-u S_{\alpha}(G_{\lambda,1}(N_{1/c}(1)))}\right)\\
& = \mathbb{E} \left(\mathbb{E} \left(e^{-u S_{\alpha}(G_{\lambda,1}(N_{1/c}(1)))}|G_{\lambda,1}(N_{1/c}(1))\right)\right)\\
& = \mathbb{E} \left( e^{-u ^{\alpha}G_{\lambda,1}(N_{1/c}(1))}\right)\\
& = \mathbb{E} \left(\mathbb{E} \left(e^{-u^{\alpha} G_{\lambda,1}(N_{1/c}(1))}|N_{1/c}(1)\right)\right)\\
& = \mathbb{E} \left( \frac{\lambda}{\lambda + u^{\alpha}}\right)^{N_{1/c}(1)}= \mathbb{E} \left( \frac{\lambda}{\lambda + u^{\alpha}}\right)^{1+ NB_{1/c}(1)}\\
& = \sum_{j=0}^{\infty} \left(\frac{\lambda}{\lambda + u^{\alpha}}\right)^{1+j} \frac{1}{c}\left(1- \frac{1}{c}\right)^j\;\;(\mbox{using}\;\eqref{NB-marginal})\\
& = \frac{\lambda}{\lambda + cu^{\alpha}},
\end{align*}
which is equal to $\mathbb{E}(e^{-c^{1/\alpha}u M_{\alpha,\lambda}(1)})$ and hence the result.
\end{proof} 

\setcounter{equation}{0}
\section{Simulation}

The sample paths of MLLP can be simulated by subordinating a discretized gamma L\'evy process with stable process. We first simulate a discretized gamma process on equally spaced intervals. Then we simulate the increments of MLLP process $Y_k = M_{\alpha, \lambda}(t_k) -
M_{\alpha, \lambda}(t_{k-1})$ which conditionally on the values of $G_{\lambda,1}(t_k)$ and $G_{\lambda,1}(t_{k-1})$ is a stationary stable process. Note that 
\begin{align*}
Y_k = M_{\alpha, \lambda}(t_k) - M_{\alpha, \lambda}(t_{k-1})& \stackrel{d} = S_{\alpha}(G_{\lambda,1}(t_k)) - S_{\alpha}(G_{\lambda,1}(t_{k-1})) \\
& \stackrel{d} = S_{\alpha}(G_{\lambda,1}(t_k) - G_{\lambda,1}(t_{k-1}))\\
& \stackrel{d}=  S_{\alpha}(G_{\lambda,1}(t_k - t_{k-1})) \stackrel{d} = (G_{\lambda,1}(t_k - t_{k-1}))^{1/\alpha} S_{\alpha}(1).
\end{align*}
Using the above ideas, we have following algorithm to simulate the sample paths of MLLP. We use Python 3.5.1, matplotlib 1.5.1 and numpy 1.10.4 for sample paths simulation.
\begin{enumerate}[(i)]
\item Choose a finite interval $[0, t]$. Suppose our aim is to have $n$-simulated values at
$t_1 = t/n, \cdots , t_{n-1} = (n - 1)t/n,$ and $t_n = t.$
\item Generate a vector $G$ of size $n$ of gamma variates such that $G = (G_1, G_2, \cdots, G_n)$, with $G_i \sim$ gamma$(\lambda, t_i- t_{i-1})$.
\item Now generate an independent vector of size $n$ of standard $\alpha$-stable random numbers $S = (S_1, S_2, \cdots, S_n)$.
\item Compute $Y_k = G_k ^{1/\alpha}S_k,\;k\geq 1$ and $M_i = \sum_{k=1}^{i} Y_k.$
Then $M_1, M_2, \cdots , M_n$ denote $n$-simulated values from MLLP.
\end{enumerate}

\begin{figure}[ht!]
\centering
 \subfloat[$\alpha = 0.9$ and $\lambda= 1.0$\label{fig:a}]{\includegraphics[width=.5\textwidth, height = .5\textwidth]{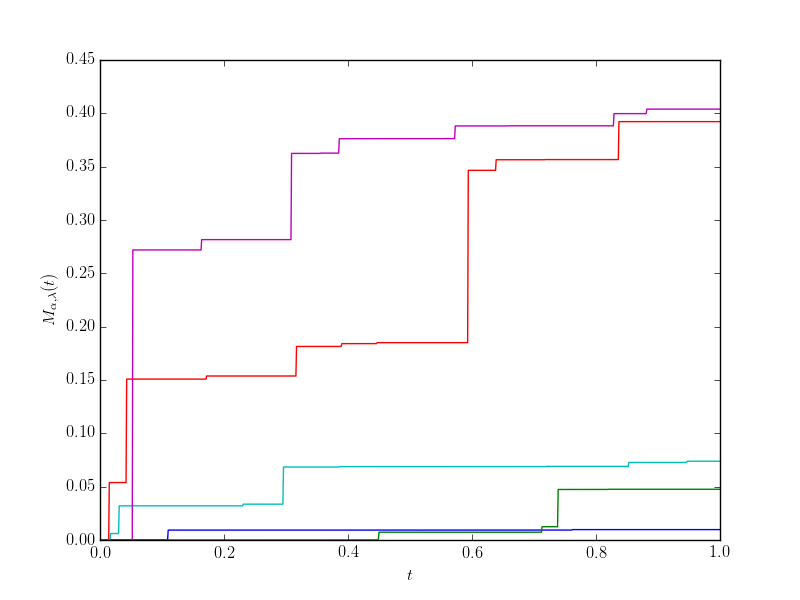}}
 \subfloat[$\alpha = 0.4$ and $\lambda= 1.0$ \label{fig:b}]{\includegraphics[width=.5\textwidth, height = .5\textwidth]{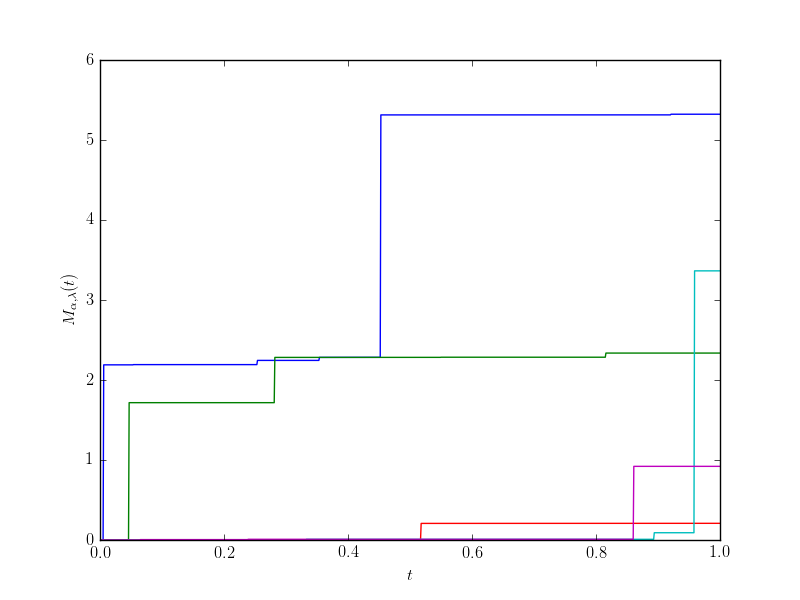}}   
\caption{Sample Paths of MLLP}
\end{figure}

\setcounter{equation}{0}
\section{Generalization}
As the integer order moments of MLLP are not finite, we can generalize MLLP to a tempered MLLP by taking subordination with respect to a tempered stable subordinator (hereafter TSS). TSS are infinitely divisible, have exponentially decaying tail probabilities and have all moments finite. These properties are obtained on the cost of self-similarity. 
Let $S_{\alpha, \mu}(t)$ be the TSS with index $\alpha \in(0,1)$ and  tempering parameter $\mu$.
TSS are obtained by exponential tempering in the distribution of stable processes (see e.g. Rosinski, 2007).
Further, the LT of density of TSS (see Meerschaert et al. 2013) is
\begin{equation}\label{tem-LT}
\mathbb{E}(e^{-uS_{\alpha, \mu}(t)}) = e^{-t\big((u+\mu)^{\alpha}-\mu^{\alpha}\big)}. 
\end{equation}
We define, tempered MLLP as follows.

\begin{definition}[Tempered MLLP]
Tempered MLLP is defined by subordinating TSS with independent gamma process such that
\begin{equation}
M^*_{\alpha, \lambda, \mu}(t): = S_{\alpha, \mu}(G_{\lambda,1}(t)),\; \alpha\in(0,1),\; \lambda>0,\;\mu>0,\;t\geq 0.
\end{equation}
\end{definition}

By using the standard conditioning argument with \eqref{tem-LT}, we obtain
\begin{equation}
\mathbb{E}(e^{-u M^*_{\alpha, \lambda, \mu}(t)}) = \left(\frac{\lambda}{\lambda - \mu^{\alpha}+ (\mu+u)^{\alpha}}\right)^t.
\end{equation}
Without proof, we state the following result for marginal density and L\'evy density for tempered MLLP. The proof is analogous to Propositions \ref{marginal-density} and \ref{levy-density}.  

\begin{proposition}
The marginal density function and L\'evy density for tempered MLLP are given by
\begin{equation*}
f_{M^*_{\alpha, \lambda, \mu}(t)}(x) = \lambda^t e^{-\mu x} \sum_{k=0}^{\infty} (\mu^{\alpha}-\lambda)^k \frac{\Gamma(t+k)}{\Gamma(k+1)\Gamma(t)} \frac{x^{\alpha(t+k)-1}}{\Gamma(\alpha(t+k))}.
\end{equation*}
Further, the L\'evy density is 
\begin{align*} \label{MLLP-LevyMeasure}
\nu^*(dx) = \frac{\alpha e^{-\mu x}}{x}E_{\alpha,1}((\mu^{\alpha}-\lambda) x^{\alpha}), x>0.
\end{align*} 
\end{proposition}

We have $\mathbb{E}G_{\lambda,1}(t) = t/\lambda$ and var$(G_{\lambda,1}(t)) = t/\lambda^2.$ Further, $\mathbb{E}(S_{\alpha,\mu}(t)) = \frac{\alpha t}{\mu^{1-\alpha}}$ and var$(S_{\alpha,\mu}(t)) = \frac{\alpha(1-\alpha)t}{\mu^{2-\alpha}}$. Using the information above and with the help of conditioning argument, we have $\mathbb{E}(M^*_{\alpha, \lambda, \mu}(t)) = \frac{\alpha t}{\lambda\mu^{1-\alpha}}$ and var$(M^*_{\alpha, \lambda, \mu}(t)) = \frac{\alpha(1-\alpha)t}{\lambda\mu^{2-\alpha}} + \frac{\alpha^2 t}{\lambda^2\mu^{2-\alpha}}. $ Moreover, for tempered MLLP all order moments are finite and tail probability decays exponentially.
\vtwo
\noindent {\bf \Large References}
\vone
\noindent
\begin{namelist}{xxx}

\item{}  Applebaum, D. 2009. { L\'evy Processes and Stochastic Calculus}. 2nd ed., {Cambridge University Press}, Cambridge, U.K.

\item {} Barndorff-Nielsen, O. E. 2000. Probability densities and L\'evy densities. Research 18, Aarhus Univ., Centre for Mathematical Physics and Stochastics (MaPhySto).

\item {} Beghin, L. and  Orsingher, E. 2009. Fractional Poisson processes and related random motions. Electron. J.
Probab., 14.  1790�-1826.

\item {} Bertoin, J. 1996. L\'evy Processes. Cambridge University Press, Cambridge.

\item {} Cahoy, D. O., Uchaikin, V. V. and Woyczynski, W. A. 2010.  Parameter estimation for fractional Poisson processes.
J. Statist. Plann. Inf. 140. 3106--3120.

\item{} Feller, W. 1971. Introduction to Probability Theory and its Applications. Vol. II. John
Wiley, New York.


\item {} Gorenflo, R., Kilbas, A. A., Mainardi, F. and Rogosin, S. V. 2014. Mittag-Leffler Functions, Related Topics and Applications. Springer, New York.

\item{} Kanter, M. 1975. Stable densities under change of scale and total variation inequlities. Annals of Probability 3, 697--707.

\item{} Kozubowski, T. J., Meerschaert, M. M. and Podgorski, K. 2006. Fractional Laplace motion. Adv. Appl. Prob. 38, 451-464.

\item {} Laskin, N. 2003. Fractional Poisson process. Commun. Nonlinear Sci. Numer. Simul. 8, 201--213.

\item{}  Meerschaert, M. M., Nane, E., Vellaisamy, P. 2013. Transient anamolous subdiffusions  on bounded domains. {Proc. Amer. Math. Soc.}, 141, 699-710.
 
\item{}  Meerschaert, M. M., Nane, E., Vellaisamy, P. 2011. The fractional Poisson process and the inverse stable subordinator.
Electron. J. Probab., 16 , 1600--1620.

\item {} Pillai, R. N. 1990. On Mittag-Leffler functions and related distributions. Ann. Inst. Statist. Math. 42, 157--161.

\item {}  Repin, O.N. and Saichev, A.I. 2000 Fractional Poisson law. Radiophys. and Quantum Electronics. 43, 738--741.

\item {} Rosi\'nski, J., 2007.  {Tempering stable processes}. {Stochastic Process  Appl.} { 117},  677--707.

\item {} Sato, K.-I. 1999. L\'evy Processes and Infinitely Divisible Distributions. Cambridge University
Press.
 \end{namelist}
\end{document}